\documentclass[smallextended,numbook]{svjour3_mod}

\usepackage{amsmath}
\usepackage{amssymb}

\usepackage{mathptmx}
\usepackage{helvet}
\usepackage{courier}

\begin{document}

\title{A Note on Approximate Inverse Iteration}

\author{Harry Yserentant}

\institute{
 Institut f\"ur Mathematik, Technische Universit\"at Berlin,
 10623 Berlin, Germany\\
\email{yserentant@math.tu-berlin.de}}

\date{November 13, 2016}

\titlerunning{A Note on Approximate Inverse Iteration}
 
\authorrunning{H. Yserentant}

\maketitle

%%%%%%%%%%%%%%%%%%%%%%%%%%%%%%%%%%%%%%%%%%%%%%%%%%%%%%%%%%%%%%%%%%%%%%%%%

\begin{abstract}
Different variants of approximate inverse iteration like the 
locally optimal block preconditioned conjugate gradient method 
became in recent years increasingly popular for the solution 
of the large matrix eigenvalue problems arising from the 
discretization of selfadjoint elliptic partial differential 
equations, in particular for the calculation of the minimum 
eigenvalue. We extend in this little note the classical 
convergence theory of D'yakonov and Orekhov [Math. Notes 27 
(1980)] to the case of operators with an essential spectrum 
on infinite dimensional Hilbert spaces and allow for arbitrary, 
sufficiently small perturbations of the solutions of the 
equation that links the iterates. The note complements the 
much more elaborate convergence theory of Neymeyr and Knyazev 
and Neymeyr for the matrix case; see [Knyazev and Neymeyr, 
SIAM J. Matrix Anal. Appl. 31 (2009)] and the references 
therein. 

%
% Keywords:
%
% Eigenvalues, Eigenfunctions, Inverse Iteration 
%
%
% Mathematics Subject Classification (MSC2010):
%


%
% 65Nxx  Partial differential equations, boundary value problems
%
% 65N25  Eigenvalue problems
%
% 65N30  Finite elements, Rayleigh-Ritz and Galerkin methods
%
% 65N15  Error bounds

\end{abstract}

%%%%%%%%%%%%%%%%%%%%%%%%%%%%%%%%%%%%%%%%%%%%%%%%%%%%%%%%%%%%%%%%%%%%%%%%%

\renewcommand {\thesection}{\arabic{section}}
\renewcommand {\theequation}{\arabic{section}.\arabic{equation}}

\renewcommand {\thelemma}{\arabic{section}.\arabic{lemma}}
\renewcommand {\thetheorem}{\arabic{section}.\arabic{theorem}}

\def \HS      {H}
\def \E       {E}

%%%%%%%%%%%%%%%%%%%%%%%%%%%%%%%%%%%%%%%%%%%%%%%%%%%%%%%%%%%%%%%%%%%%%%%%%

\section{Introduction}

Preconditioned inverse iteration evolved in recent 
years into a very popular method for the solution 
of the large matrix eigenvalue problems that arise 
from the discretization of linear selfadjoint 
elliptic partial differential equations, in particular 
in form of the locally optimal (block-)preconditioned 
conjugate gradient method \cite{Knyazev}. The analysis 
of such methods essentially started with the work of 
D'yakonov and Orekhov~\cite{Dyakonov-Orekhov} from 
the early 1980's. In a series of groundbreaking papers, 
Neymeyr \cite{Neymeyr_1,Neymeyr_2,Neymeyr_3} and
Knyazev and Neymeyr 
\cite{Knyazev-Neymeyr_1,Knyazev-Neymeyr_2} analyzed 
these methods in great detail and determined the 
exact convergence rate of their basic variant. The 
resulting estimates are short and elegant, but their 
proof is by no means simple and requires a long, 
complicated, and tedious analysis. Their 
generalization to the infinite dimensional case is 
anything but obvious and necessitates additional 
considerations~\cite{Rohwedder-Schneider-Zeiser}, 
in particular in the presence of an essential 
spectrum as in the example of the electronic 
Schr\"odinger equation, a basic equation of quantum 
physics and chemistry.
Therefore we adapt in this little note the original, 
comparatively simple and short proof of D'yakonov 
and Orekhov \cite{Dyakonov-Orekhov} to this general
situation and derive, in the language of quantum 
mechanics, a variant for the calculation of the 
ground state energy, the minimum eigenvalue, and 
an associated eigenfunction. The price to be paid
is that the initial approximation of this 
eigenfunction, from which the iterative process 
starts, must already possess a Rayleigh quotient 
below the rest of the spectrum and that the derived 
error bounds are surely not best possible. They show, 
however, qualitatively the right behavior, that is, 
they depend like the best possible bounds for the 
matrix case only on the minimum eigenvalue and
its distance to the rest of the spectrum and on a 
constant that controls the accuracy of the 
approximate solutions of the equation that links 
the iterates. Our presentation is completely based 
on the weak form of the eigenvalue problem and 
refers to it only via the assigned bilinear forms.

%%%%%%%%%%%%%%%%%%%%%%%%%%%%%%%%%%%%%%%%%%%%%%%%%%%%%%%%%%%%%%%%%%%%%%%%%

\section{Approximate inverse iteration in a general setting}

\setcounter{equation}{0}

Let $\HS$ be a Hilbert space that is equipped with 
the inner product $a(u,v)$ inducing the energy norm 
$\|u\|$, under which it is complete, and a further 
inner product $(u,v)$ that induces the weaker norm 
$\|u\|_0$. Let the infimum of the Rayleigh quotient 
\begin{equation}    \label{eq2.1}
\lambda(u)=\frac{a(u,u)}{(u,u)}, \quad 
\text{$u\neq 0$ in $\HS$},
\end{equation}
be an isolated eigenvalue $\lambda_1>0$ of finite 
multiplicity, which means in particular that the norm
$\|u\|_0$ of an element $u\in\HS$ can be estimated 
by its energy norm $\|u\|$. Let $\E_1$ be the 
corresponding eigenspace, the finite dimensional 
space of all $u\in\HS$ for which 
\begin{equation}    \label{eq2.2}
a(u,\chi)=\lambda_1(u,\chi), \quad \chi\in\HS,
\end{equation}
or equivalently $\lambda(u)=\lambda_1$ 
holds. Our primary aim is the calculation of this
eigenvalue, in quantum mechanics the ground state 
energy of the system under consideration, and to
a lesser degree also of an eigenvector for this 
eigenvalue.

As $a(u,\chi_1)=\lambda_1(u,\chi_1)$ for all 
$\chi_1\in\E_1$ and as $\lambda_1\neq 0$, the 
orthogonal complement
\begin{equation}    \label{eq2.3}
\E_1^\bot=
\{u\in\HS\,|\,\text{$(u,\chi_1)=0$ for all $\chi_1\in\E_1$}\}
\end{equation}
of $\E_1$ with respect to the inner product $(u,v)$ 
is at the same time the $a$-orthogonal complement
of this eigenspace. Let $\lambda_2$ be the infimum of 
the Rayleigh quotient on $\E_1^\bot$, itself a point 
in the spectrum. Since $\lambda_1$ is an isolated 
eigenvalue, $\lambda_2>\lambda_1$. In most cases,  
$\lambda_2$ will also be an isolated eigenvalue,  
even in the presence of an essential spectrum as
in the example of the electronic Schr\"odinger 
equation, but neither this nor any other additional 
assumption on the structure of the spectrum will 
enter into our argumentation.

Given an element $u\in\HS$~with norm $\|u\|_0=1$ and 
Rayleigh quotient $\lambda(u)<\lambda_2$, in inverse 
iteration in its original, exact version at first 
the solution $w\in\HS$ of the equation
\begin{equation}    \label{eq2.4}
a(w,\chi)=a(u,\chi)-\lambda(u)(u,\chi), \quad \chi\in\HS,
\end{equation}
is determined, which exists by the Lax-Milgram  
or the Riesz representation theorem and is unique. 
The current $u$ is then replaced by $u-w$. Since 
$a(u,w)=0$,
\begin{equation}    \label{eq2.5}
\|u-w\|^2=\|u\|^2+\|w\|^2.
\end{equation} 
The new element $u-w$ is thus different from zero so 
that $\lambda(u-w)$ is well defined and the process 
can be repeated with the normed version of $u-w$. 
The in this way iteratively generated sequence of 
Rayleigh quotients decreases then monotonously to 
the eigenvalue $\lambda_1$ and the iterates $u$ 
converge to an eigenvector for this eigenvalue.

This process can hardly be realized when $\HS$ is 
infinite dimensional or of high finite dimension. 
In the modification considered in this note, the 
solution $w$ of equation (\ref{eq2.4}) is 
therefore replaced by an approximation $v$ for 
which an error estimate
\begin{equation}    \label{eq2.6}
\|v-w\|\leq\eta\|w\|
\end{equation}
holds, where $\eta<1$ is a fixed constant that 
controls the accuracy. By (\ref{eq2.5}) and 
(\ref{eq2.6}) 
\begin{equation}    \label{eq2.7}
\|u-v\|\geq(1-\eta)\|u-w\|,
\end{equation}
so that also $u-v\neq 0$ and the whole process 
can proceed with the new iterate
\begin{equation}    \label{eq2.8}
u'=\frac{u-v\,}{\|u-v\|_0}.
\end{equation}
We do not make any assumption on the origin of 
$v$. It can, for example, be the element of best 
approximation of $w$ in an appropriately chosen 
finite dimensional subspace of~$\HS$, an 
iteratively calculated approximation of this 
element, or anything else. We will analyze in 
the next section the convergence properties of 
this general form of approximate inverse 
iteration along the lines given by D'yakonov 
and Orekhov.

%%%%%%%%%%%%%%%%%%%%%%%%%%%%%%%%%%%%%%%%%%%%%%%%%%%%%%%%%%%%%%%%%%%%%%%%%

\section{Convergence and error estimates}

\setcounter{equation}{0}

Starting point of our analysis is as described an 
element $u\in\HS$ of norm $\|u\|_0=1$ with Rayleigh
quotient $\lambda(u)<\lambda_2$. We assume that $w$ 
is the unique solution of equation (\ref{eq2.4}), 
$v$~an approximation of this $w$ satisfying 
(\ref{eq2.6}), and $u'$ the normed version 
(\ref{eq2.8}) of $u-v$. For abbreviation, we set 
$\lambda=\lambda(u)$ and $\lambda'=\lambda(u')$. 
Moreover, we need the $a$-orthogonal projection 
$P_1$ of the Hilbert space $\HS$ onto the eigenspace 
$\E_1$ for the eigenvalue $\lambda_1$, which is at 
the same time the orthogonal projection of $\HS$ 
to $\E_1$ with respect to the other inner product 
$(u,v)$, and the with respect to both inner products 
orthogonal projection $Q=I-P_1$ of $\HS$ onto the 
orthogonal complement (\ref{eq2.3}) of $\E_1$. 

\begin{lemma}       \label{lm3.1}
The energy norm of $w$ can be estimated from below 
as follows:
\begin{equation}    \label{eq3.1}
\|w\|^2\,\geq\,
\bigg(\frac{\lambda_2-\lambda}{\lambda_2}\bigg)^2
(\lambda-\lambda_1).
\end{equation}
\end{lemma}

\begin{proof}
For $f\in\HS$, let $Gf\in\HS$ be the solution of 
the equation 
\begin{displaymath}
a(Gf,\chi)=(f,\chi), \quad \chi\in\HS.
\end{displaymath}
The element $w=u-\lambda Gu$ is then the solution 
of the equation (\ref{eq2.4}). Thus 
\begin{displaymath}
\|w\|\geq\|Qu-\lambda QGu\|.
\end{displaymath}
As $Qu$ and $QGu$ are in $\E_1^\bot$ and $Q$ is 
orthogonal with respect to both inner products,
\begin{displaymath}
\|QGu\|^2=(Qu,QGu)\leq\|Qu\|_0\|QGu\|_0\leq
\lambda_2^{-1}\|Qu\|\|QGu\|.
\end{displaymath}
Therefore $\|QGu\|\leq\lambda_2^{-1}\|Qu\|$. 
As $1-\lambda\lambda_2^{-1}>0$, this yields
\begin{displaymath}
\|w\|^2\,\geq\,
\bigg(\frac{\lambda_2-\lambda}{\lambda_2}\bigg)^2\|Qu\|^2.
\end{displaymath}
As $P_1u$ and $Qu$ are orthogonal to each other
and $\|u\|_0=1$,
\begin{displaymath}
\|Qu\|^2=\|u\|^2-\|P_1u\|^2=\lambda-\lambda_1\|P_1u\|_0^2
\,\geq\lambda-\lambda_1,
\end{displaymath}
from which the estimate (\ref{eq3.1}) finally follows.
\qed
\end{proof}

\begin{lemma}       \label{lm3.2}
The distance of $\lambda$ and $\lambda'$ can be estimated 
from below as
\begin{equation}    \label{eq3.2}
\lambda-\lambda'\,\geq\,
\frac{\lambda\,(1-\eta^2)\|w\|^2}{\lambda+(1-\eta^2)\|w\|^2}.
\end{equation}
\end{lemma}

\begin{proof}
As $(u,u)=1$, $a(u,u)=\lambda$, and because $w$ is a 
solution of equation (\ref{eq2.4}),
\begin{displaymath}
\lambda-\lambda'\,=\,
\frac{a(w,w)-a(v-w,v-w)+\lambda(v,v)}{1-2(u,v)+(v,v)}.
\end{displaymath}
Estimating the denominator with help of the assumption
(\ref{eq2.6}) on the accuracy of $v$ from below and 
the nominator using $\|u\|_0=1$ from above, one obtains 
the estimate
\begin{displaymath}
\lambda-\lambda'\,\geq\,
\frac{(1-\eta^2)\|w\|^2+\lambda\,\|v\|_0^2\,}
{\,1+2\,\|v\|_0+\|v\|_0^2}.
\end{displaymath}
The right hand side becomes, as a function of 
the norm $\|v\|_0$, minimal if
\begin{displaymath}
\lambda\,\|v\|_0=(1-\eta^2)\|w\|^2
\end{displaymath} 
and attains then the value on the right hand 
side of (\ref{eq3.2}).
\qed
\end{proof}

\begin{theorem}     \label{thm3.1}
Under the given assumptions, and if in particular 
$\lambda(u)<\lambda_2$, 
\begin{equation}    \label{eq3.3}
\lambda(u')-\lambda_1\,\leq\,q(\lambda(u))(\lambda(u)-\lambda_1)
\end{equation}
holds, where $q(\lambda)$ is the on the interval 
$\lambda_1\leq\lambda\leq\lambda_2$ strictly
increasing function
\begin{equation}    \label{eq3.4}
q(\lambda)\,=\;1-\,
\frac{(1-\eta^2)\,\lambda\,(\lambda_2-\lambda)^2}
{\lambda_2^2\,\lambda+(1-\eta^2)(\lambda_2-\lambda)^2(\lambda-\lambda_1)}.
\end{equation}
\end{theorem}

\begin{proof}
The function $x\to\lambda x/(\lambda+x)$ is monotonously 
increasing. If one inserts the estimate (\ref{eq3.1}) into 
the estimate (\ref{eq3.2}), one obtains therefore the 
lower bound
\begin{displaymath}
\lambda-\lambda'\,\geq\,
\frac{(1-\eta^2)\,\lambda\,(\lambda_2-\lambda)^2(\lambda-\lambda_1)}
{\lambda_2^2\,\lambda+(1-\eta^2)(\lambda_2-\lambda)^2(\lambda-\lambda_1)}
\end{displaymath}
for the difference $\lambda-\lambda'$ of the two Rayleigh
quotients. If $\lambda>\lambda_1$, the representation
\begin{displaymath}
\lambda'-\lambda_1\,=\,
\bigg(1-\frac{\lambda-\lambda'}{\lambda-\lambda_1}\,\bigg)
(\lambda-\lambda_1)
\end{displaymath}
of $\lambda'-\lambda_1$ thus yields the estimate 
(\ref{eq3.3}). The function $q(\lambda)$ possesses 
the derivative
\begin{displaymath}
q'(\lambda)\,=\;\frac{(1-\eta^2)(\lambda_2-\lambda)
\big((1-\eta^2)\,\lambda_1(\lambda_2-\lambda)^3+2\lambda_2^2\,\lambda^2\big)}
{\big(
\lambda_2^2\,\lambda+(1-\eta^2)(\lambda_2-\lambda)^2(\lambda-\lambda_1)
\big)^2}
\end{displaymath}
and is therefore strictly increasing on the interval 
under consideration. If $\lambda=\lambda_1$ and $u$ is 
therefore an eigenvector for the eigenvalue~$\lambda_1$,
$w=0$ and thus also $v=0$. The iteration comes then to a 
halt, it is $u'=u$ and $\lambda'=\lambda_1$, and 
(\ref{eq3.3}) trivially holds.
\qed
\end{proof}

\begin{lemma}     \label{lm3.3}
The energy norm of the approximation $v$ of the solution 
$w$ of equation {\rm (\ref{eq2.4})} can be estimated in 
terms of the distance of the Rayleigh quotient $\lambda$ 
to $\lambda_1$:
\begin{equation}    \label{eq3.5}
\|v\|^2\,\leq\,
\frac{1+\eta}{1-\eta}\,\frac{\lambda_2}{\lambda_1}\,
(\lambda-\lambda_1).
\end{equation}
\end{lemma}

\begin{proof}
The estimate from Lemma~\ref{lm3.2} is equivalent to 
\begin{displaymath}
\|w\|^2\,\leq\,\frac{1}{1-\eta^2}\,\frac{\lambda}{\lambda'}\,
(\lambda-\lambda').
\end{displaymath}
As $\|v\|\leq(1+\eta)\|w\|$ and $\lambda_1\leq\lambda'$
and $\lambda\leq\lambda_2$, the inequality (\ref{eq3.5})
follows.
\qed
\end{proof}
If one repeats the process, the Rayleigh quotients 
approach by (\ref{eq3.3}) the eigenvalue~$\lambda_1$. 
At the same time, the norm of the assigned vectors 
$v$ tends by (\ref{eq3.5}) to zero.
\begin{lemma}       \label{lm3.4}
If already $\|v\|^2\leq\lambda_1/4$, the distance of 
the normed version $u'$ of the new vector $u-v$ and 
the given normed $u$ satisfies an estimate
\begin{equation}    \label{eq3.6}
\|u-u'\|^2\,\leq\,
c(\lambda_2/\lambda_1,\eta)(\lambda-\lambda_1),
\end{equation}
with a constant depending only on $\eta$ and the 
ratio $\lambda_2/\lambda_1$
\end{lemma}
\begin{proof}
Because $u$ has the norm $\|u\|_0=1$, the energy norm 
of $u-u'$ can be written as
\begin{displaymath}
\|u-u'\|\,=\,\frac{\|\,v-(\|u\|_0-\|u-v\|_0)\,u\,\|}{\|u-v\|_0}.
\end{displaymath}
The triangle inequality leads therefore to the 
estimate
\begin{displaymath}
\|u-u'\|\leq\frac{\|v\|+\|v\|_0\|u\|}{1-\|v\|_0}
\end{displaymath}
or, because of 
$\|v\|_0^2\leq\lambda_1^{-1}\|v\|^2\leq 1/4$ and
$\|u\|^2=\lambda(u)$, $\lambda(u)\leq\lambda_2$, to
\begin{displaymath}
\|u-u'\|\,\leq\, 2\;
\bigg(1+\Big(\frac{\lambda_2}{\lambda_1}\Big)^{\!1/2\,}\bigg)\|v\|.
\end{displaymath}
Estimating the norm of $v$ by (\ref{eq3.5}) in terms
of $\lambda(u)-\lambda_1$, the proposition follows.
\qed
\end{proof}

If one starts therefore with a normed initial 
approximation $u=u_0$ in $\HS$ with Rayleigh 
quotient $\lambda(u_0)<\lambda_2$ and generates 
in the manner described a sequence of normed~$u_k$, the
Rayleigh quotients $\lambda(u_k)$ decrease 
strictly to the minimum eigenvalue $\lambda_1$ 
or become stationary there and the estimate
\begin{equation}    \label{eq3.7}
\lambda(u_{k+1})-\lambda_1\,\leq\, 
q(\lambda(u_k))(\lambda(u_k)-\lambda_1)
\end{equation}
holds. As 
$q(\lambda(u_k))\leq q(\lambda(u_0))$,
\begin{equation}    \label{eq3.8}
\lambda(u_k)-\lambda_1\,\leq\, 
q(\lambda(u_0))^k(\lambda(u_0)-\lambda_1),
\end{equation}
so that the Rayleigh quotients converge because of 
$q(\lambda(u_0))<q(\lambda_2)=1$ rapidly to their 
limit. When $k$ goes to infinity, the error 
reduction factors $q(\lambda(u_k))$ fall to 
\begin{equation}    \label{eq3.9}
q(\lambda_1)\,=\,1-\,(1-\eta^2)
\bigg(\frac{\lambda_2-\lambda_1}{\lambda_2}\bigg)^{\!2},
\end{equation}
a value less surprisingly worse than the 
optimal limit value
\begin{equation}    \label{eq3.10}
\bigg(1-(1-\eta)\,\frac{\lambda_2-\lambda_1}{\lambda_2}\bigg)^2
\end{equation}
that Knyazev and Neymeyr \cite{Knyazev-Neymeyr_1}
obtain for the matrix case. For sufficiently large 
$k$, when the norms of the assigned $v_k$ are 
already sufficiently small, by Lemma~\ref{lm3.4} 
\begin{equation}    \label{eq3.11}
\|u_k-u_{k+1}\|^2\,\leq\, c\,(\lambda(u_k)-\lambda_1)
\end{equation}
holds with some constant $c$ that depends only
on $\eta$ and the ratio $\lambda_2/\lambda_1$.
In view of the error estimate (\ref{eq3.8}) for
the Rayleigh quotients, the $u_k$ thus form a 
Cauchy sequence in the Hilbert space $\HS$.
They converge therefore to a limit $u^*$ of 
norm $\|u^*\|_0=1$. 

Our final theorem shows that the distance of 
an arbitrary normed $u$ with Rayleigh quotient less 
than $\lambda_2$ to its best approximation by an 
element in the eigenspace $\E_1$ can be bounded 
in terms of the distance of its Rayleigh quotient 
to the eigenvalue $\lambda_1$.

\begin{theorem}     \label{thm3.2}
For all $u$ of norm $\|u\|_0=1$ with Rayleigh quotient
$\lambda(u)<\lambda_2$,
\begin{equation}    \label{eq3.12}
\|u-P_1u\|^2\,\leq\,
\bigg(\frac{\lambda_2-\lambda_1}{\lambda_2}\bigg)^{\!-1}\!
(\lambda(u)-\lambda_1).
\end{equation}
\end{theorem}

\begin{proof}
Let again $\lambda=\lambda(u)$ for abbreviation.
The proof is based on the relation
\begin{displaymath}
0\,=\,\|u\|^2-\lambda\,\|u\|_0^2\,=\,
\|P_1u\|^2-\lambda\,\|P_1u\|_0^2+\|Qu\|^2-\lambda\,\|Qu\|_0^2.
\end{displaymath}
Using $\|P_1u\|^2=\lambda_1\|P_1u\|_0^2$,
$\|Qu\|^2\geq\lambda_2\|Qu\|_0^2$,
and moreover that
\begin{displaymath}
\|Qu\|_0^2=1-\|P_1u\|_0^2,
\end{displaymath}
one obtains from this relation the lower estimate
\begin{displaymath}
\|P_1u\|_0^2\,\geq\,
\frac{\lambda_2-\lambda}{\,\lambda_2-\lambda_1}
\end{displaymath}
for $P_1u$. 
Since $\|u-P_1u\|^2=\lambda-\lambda_1\|P_1u\|_0^2$, 
this proves the estimate (\ref{eq3.12}).
\qed
\end{proof}
This shows in particular that $u^*=P_1u^*$, so that 
the in the course of the iteration generated vectors 
$u_k$ converge indeed to an eigenvector for the 
minimum eigenvalue.

%%%%%%%%%%%%%%%%%%%%%%%%%%%%%%%%%%%%%%%%%%%%%%%%%%%%%%%%%%%%%%%%%%%%%%%%%

%%%%%%%%%%%%%%%%%%%%%%%%%%%%%%%%%%%%%%%%%%%%%%%%%%%%%%%%%%%%%%%%%%%%%%%%%


\begin{thebibliography}{1}

\providecommand{\url}[1]{{#1}}
\providecommand{\urlprefix}{URL }
\expandafter\ifx\csname urlstyle\endcsname\relax
  \providecommand{\doi}[1]{DOI~\discretionary{}{}{}#1}\else
  \providecommand{\doi}{DOI~\discretionary{}{}{}\begingroup
  \urlstyle{rm}\Url}\fi

\bibitem{Dyakonov-Orekhov}
D'yakonov, E., Orekhov, M.: Minimization of the computational labor 
  in determining the first eigenvalues of differential operators.
\newblock Mat. Zametki \textbf{27}, 795--812 (1980).
\newblock In {R}ussian, {E}nglish translation: Math. Notes 27 (1980), 
 pp. 382--391

\bibitem{Knyazev}
Knyazev, A.V.: Toward the optimal preconditioned eigensolver: 
  locally optimal block preconditioned conjugate gradient method.
\newblock SIAM J. Sci. Comput. \textbf{23}, 517--541 (2001)

\bibitem{Knyazev-Neymeyr_1}
Knyazev, A.V., Neymeyr, K.: A geometric theory for preconditioned 
  inverse iteration. {III}. {A} short and sharp convergence 
  estimate for generalized eigenvalue problems.
\newblock Linear Algebra Appl. \textbf{358}, 95--114 (2003)

\bibitem{Knyazev-Neymeyr_2}
Knyazev, A.V., Neymeyr, K.: Gradient flow approach to geometric 
  convergence analysis of preconditioned eigensolvers.
\newblock SIAM J. Matrix Anal. Appl. \textbf{31}, 621--628 (2009)

\bibitem{Neymeyr_1}
Neymeyr, K.: A geometric theory for preconditioned inverse 
  iteration. {I}. {E}xtrema of the {R}ayleigh quotient.
\newblock Linear Algebra Appl. \textbf{322}, 61--85 (2001)

\bibitem{Neymeyr_2}
Neymeyr, K.: A geometric theory for preconditioned inverse 
  iteration. {II}. {C}onvergence estimates.
\newblock Linear Algebra Appl. \textbf{322}, 87--104 (2001)

\bibitem{Neymeyr_3}
Neymeyr, K.: A geometric theory for preconditioned inverse 
  iteration. {IV}. {O}n the fastest convergence cases.
\newblock Linear Algebra Appl. \textbf{415}, 114--139 (2006)

\bibitem{Rohwedder-Schneider-Zeiser}
Rohwedder, T., Schneider, R., Zeiser, A.: Perturbed 
  preconditioned inverse iteration for operator 
  eigenvalue problems with applications to adaptiv
  wavelet discretization.
\newblock Adv. Comput. Math. \textbf{34}, 43--66 (2011)

\end{thebibliography}
\end{document}